\documentclass[11pt, a4paper]{amsart}

\usepackage{graphicx}
\usepackage{color}
\usepackage{nicefrac}
\def\ep{\varepsilon}
\def\O{\Omega}
\def\di{\displaystyle}
\def\R{{\mathbb {R}}}

\def\lp{L^{p(\cdot)}(\Omega)}
\def\lq{L^{q(\cdot)}(\Omega)}

\def\lpe{L^{p^*(\cdot)}(\Omega)}
\def\wp{W^{1,p(\cdot)}(\Omega)}

\def\Lip{\mathop{\mbox{\normalfont Lip}}\nolimits(\O)}
\def\Lipc{\mathop{\mbox{\normalfont Lip}}\nolimits(\overline{\O})}

\def\M{\mathcal{M}}
\newtheorem{teo}{Theorem}[section]
\newtheorem{lema}[teo]{Lemma}

\newtheorem{corol}[teo]{Corollary}

\theoremstyle{definition}

\theoremstyle{remark}
\newtheorem{remark}[teo]{Remark}

\renewcommand{\theequation}{\arabic{section}.\arabic{equation}}

\def\div{\mathop{\mbox{\normalfont div}}\nolimits}

\begin{document}

\title[$H^2$ regularity for the $p(x)-$Laplacian in two-dimensional]{$H^2$ regularity for the $p(x)-$Laplacian in two-dimensional convex domains}

\author[L. M. Del Pezzo \& S. Mart\'{i}nez ]
{Leandro M. Del Pezzo and Sandra Mart\'{\i}nez}

\address{Leandro M. Del Pezzo \hfill\break\indent
CONICET and Departamento  de Matem\'atica, FCEyN, UBA,
\hfill\break\indent Pabell\'on I, Ciudad Universitaria (1428),
Buenos Aires, Argentina.}
\email{{\tt ldpezzo@dm.uba.ar}\hfill\break\indent {\it Web page:} {\tt http://cms.dm.uba.ar/Members/ldpezzo}}

\address{Sandra Mart\'{\i}nez \hfill\break\indent
IMAS-CONICET and Departamento  de Matem\'atica, FCEyN, UBA,
\hfill\break\indent Pabell\'on I, Ciudad Universitaria (1428),Buenos Aires, Argentina.}
\email{{\tt smartin@dm.uba.ar}}

\thanks{Supported by UBA X117, UBA 20020090300113, CONICET PIP 2009 845/10 and PIP 11220090100625.}


\maketitle
\begin{abstract}
    In this paper we study the $H^2$ global regularity for solutions  of the $p(x)-$Laplacian in two dimensional convex domains with Dirichlet boundary
    conditions. Here $p:\O \to [p_1,\infty)$ with $p\in \Lipc$  and $p_1>1$.
\end{abstract}

\section{Introduction}
Let $\Omega$ be a bounded domain in $\R^2$ and let $p:\Omega\to (1,+\infty)$ be a measurable function. In this work, we study the $H^2$ global
regularity of the weak solution of the following problem
\begin{equation}
    \begin{cases}
       - \Delta_{p(x)} u = f & \mbox{in } \O,\\
         u = g & \mbox{on } \partial \Omega,
    \end{cases}
    \label{problema}
\end{equation}
where $\Delta_{p(x)}u=\div(|\nabla u|^{p(x)-2}\nabla u)$ is the $p(x)-$Laplacian. The hypothesis over
$p$, $f$ and $g$  will be specified later.

\medskip
Note that, the $p(x)-$Laplacian extends the classical Laplacian ($p(x)\equiv2$)
and the $p-$Laplacian ($p(x)\equiv p$ with $1<p<+\infty$). This operator has been
recently used in image processing and in the modeling of electrorheological
fluids, see \cite{BCE,CLR,R}.

\medskip
Motivate by the applications to image processing problem, in
\cite{DLM}, the authors study two numerical methods to approximate
solutions of the type of \eqref{problema}. In Theorem 5.1, the
authors prove the convergence in $\wp$ of  the conformal Galerking
finite element method. It is of our interest to study, in a future
work, the rate of this convergence. In general, all the error bounds
depend on the global regularity of the second derivatives of the
solutions, see for example
 \cite{Ci,LB}. However, there appear to be no existing regularity results in the literature that can be applied here,
since all the results have either a first order or local character.
\medskip

The $H^2$ global regularity for solutions of the $p-$Laplacian
is studied in \cite{LB}. There the authors prove the following: Let $1<p\leq 2,$
$g\in H^2(\Omega),$ $f\in L^q(\O)$ ($q>2$) and $u$ be the unique weak solution of
\eqref{problema}. Then
\begin{itemize}
    \item If $\partial\O\in C^2$ then $u\in H^2(\Omega);$
    \item If $\Omega$ is convex and $g=0$ then $u\in H^2(\O);$
    \item If $\O$ is convex with a polygonal boundary and $g\equiv0$ then
        $u\in C^{1,\alpha}(\overline{\O})$ for some $\alpha\in(0,1).$
\end{itemize}

\medskip

Regarding the regularity of the weak solution of \eqref{problema}
when $f=0$, in \cite{BN, CM}, the authors prove the
$C^{1,\alpha}_{loc}$  regularity (in the scalar case and also in the
vectorial case). Then, in the paper  \cite{ELM} the authors study
the case where the functional has the so called  $(p,q)-$ growth
conditions. Following these ideas, in \cite{FanGlobal}, the author
proves that the solutions of \eqref{problema} are in
$C^{1,\alpha}(\overline{\O})$ for some $\alpha>0$ if $\O$ is a
bounded domain in $\R^N$ ($N\ge2$) with $C^{1,\gamma}$ boundary,
$p(x)$ is a H\"{o}lder function,  $f\in L^{\infty}(\O)$  and $g\in
C^{1,\gamma}(\overline{\O}).$ While in \cite{ChLy}, the authors
prove that the solutions are in $H^2_{loc}(\{x\in\O\colon p(x)\le
2\})$ if $p(x)$ is uniformly Lipschitz ($\Lip$) and $ f\in
W^{1,q(\cdot)}_{loc}(\O)\cap L^{\infty}(\O).$

\bigskip

Our aim, it is to generalized the results of \cite{LB} in the case
where $p(x)$ is a measurable function. To this end, we will need some hypothesis over the regularity of $p(x)$.
Moreover, in all our result we can avoid the restriction $g=0$, assuming some regularity of $g(x)$.

On the other hand, to prove our results, we can assume weaker conditions over the function $f$ than the ones on
\cite{ChLy}. Since, we only assume that $f\in\lq,$ we do not have a priori that the solutions are in $C^{1,\alpha}(\O).$
Then we can not use it to prove the $H^2$ global regularity.
Nevertheless, we can prove that the solutions are in $C^{1,\alpha}(\overline{\O}),$
after proving the $H^2$ global regularity.



\medskip

The main results of this paper are:

\begin{teo}\label{H2global}
Let $\O$ be a bounded domain in $\R^2$ with $C^2$ boundary, $p\in
\Lipc$ with $p(x)\ge p_1>1$, $g\in H^2(\O)$ and $u$ be the weak solution  of \eqref{problema}.
If
\begin{enumerate}
\item[(F1)]
$f\in L^{q(x)}(\O)$ with $q(x)\ge q_1> 2$ in the set $\{x\in\O\colon p(x)\le2\};$
\item[(F2)] $f\equiv0$ in the set $\{x\in\O\colon p(x)>2\}$.
\end{enumerate}
then $u\in H^2(\O).$
\end{teo}

\begin{teo}\label{H2convexa}
Let $\O$ be a bounded domain in $\R^2$ with convex boundary, $p\in\Lipc$ with $p(x)\ge p_1>1$, $g\in H^2(\O)$ and $u$ be the weak solution  of \eqref{problema}.
If $f$ satisfies $(F1)$ and $(F2)$ then $u\in H^2(\O).$
\end{teo}

Using the above theorem we can prove the following,

\begin{corol}\label{c1alpha}Let $\O$ be a bounded convex domain in $\R^2$  with polygonal boundary, $p$ and $f$ as in the previous theorem,
 $g\in W^{2,q(x)}(\O)$  and $u$ be the weak solution  of \eqref{problema}
then $u\in C^{1,\alpha}(\overline{\O})$ for some $0<\alpha<1$.
\end{corol}

Observe that this result extends the one in \cite{FanGlobal} in the case where $\O$ is a polygonal domain in $\R^2$.

\bigskip

\noindent{\bf Organization of the paper}. The rest of the paper is organized as follows. After a short Section \ref{A} where we collect some preliminaries results,
  in Section \ref{nodeg}, we
study the $H^2-$ regularity for the non-degenerated problem. In Section \ref{pt11} we prove Theorem
\ref{H2global}. Then, in Section \ref{convex}, we study the regularity of the solution $u$ of \eqref{problema} if $\O$
 is convex. In Section \ref{coments}, we make some comments on the dependence  of the $H^2-$norm of $u$ on
$p_1.$  Lastly, in Appendices \ref{B} and \ref{C} we give some results related to elliptic linear equation with bounded coefficients and Lipschitz functions, respectively.

\section{Preliminaries}\label{A}

We now introduce the space  $L^{p(\cdot)}(\Omega)$ and $W^{1,p(\cdot)}(\Omega)$
and state some of their properties.

\medskip

Let $\O$ be a bounded open set of $\R^n$ and $p \colon\Omega \to  [1,+\infty)$ be a measurable bounded function,
called a variable exponent on $\Omega$ and denote $p_{1}:= ess inf \,p(x)$ and $p_{2} := ess sup \,p(x).$

We define the variable exponent Lebesgue space $\lp$ to consist of all measurable functions
$u \colon\Omega \to \R$ for which the modular
$$
\varrho_{p(\cdot)}(u) := \int_{\Omega} |u(x)|^{p(x)}\, dx
$$
is finite. We define the Luxemburg norm on this space by
$$
\|u\|_{L^{p(\cdot)}(\Omega)}:= \inf\{k > 0\colon \varrho_{p(\cdot)}(u/k)\leq 1 \}.
$$
This norm makes $L^{p(\cdot)}(\Omega)$ a Banach space.

\medskip

For the proofs of the following theorems, we refer the reader to
\cite{DHHR}.

\begin{teo}[H\"older's inequality]
    Let $p,q, s:\O\to[1,+\infty]$ be a measurable functions such that
    \[
    \frac1{p(x)}+\frac1{q(x)} = \frac1{s(x)} \quad \mbox{in } \O.
    \]
     Then the inequality
    \[
    \|f g\|_{L^{s(\cdot)}(\O)} \le 2 \|f\|_{\lp}\|g\|_{\lq}
    \]
    for all $f\in \lp$ and $g\in\lq $
\end{teo}

Let $W^{1,p(\cdot)}(\Omega)$ denote the space of measurable functions $u$ such that, $u
$ and the distributional derivative $\nabla u$ are in $L^{p(\cdot)}(\Omega)$. The norm
$$
\|u\|_{W^{1,p(\cdot)}(\O)}:= \|u\|_{p(\cdot)} + \| |\nabla u| \|_{p(\cdot)}
$$
makes $W^{1,p(\cdot)}(\Omega)$ a Banach space.

\begin{teo}\label{ref}
Let $\di p'(x)$ such that, ${1}/{p(x)}+{1}/{p'(x)}=1.$ Then
$L^{p'(\cdot)}(\Omega)$ is the dual of $L^{p(\cdot)}(\Omega)$.
Moreover, if $p_1>1$, $L^{p(\cdot)}(\Omega)$ and $W^{1,p(\cdot)}(\Omega)$ are
reflexive.
\end{teo}

\medskip

We define the space $W_0^{1,p(\cdot)}(\Omega)$ as the closure of the
$C_0^{\infty}(\Omega)$ in $W^{1,p(\cdot)}(\Omega)$.
Then we have the following version of Poincare's inequity
(see Theorem 3.10 in \cite{KR}).

\begin{lema}[Poincare's inequity]
\label{poinc} If $p:\O\to[1,+\infty)$ is continuous in $\overline{\Omega}$,
    there exists a constant $C$ such that for every $u\in W_0^{1,p(\cdot)}(\Omega)$,
$$
\|u\|_{L^{p(\cdot)}(\Omega)}\leq C\|\nabla
u\|_{L^{p(\cdot)}(\Omega)}.
$$
\end{lema}

\medskip

In order to have better properties of these spaces, we need more hypotheses on the regularity of $p(x)$.

We say that $p$ is \emph{$\log$-H\"{o}lder continuous} in $\O$ if there exists a
constant $C_{log}$ such that
$$|p(x) - p(y)| \leq \frac{C_{log}}{\log\,\left(e+\frac{1}{|x - y|}\right)}\quad \forall\, x,y\in\O.$$

It was proved in  \cite{D}, Theorem 3.7, that if one assumes that $p$ is log-H\"{o}lder continuous then  $C^{\infty}(\bar{\Omega})$ is dense in
$W^{1,p(\cdot)}(\Omega)$ (see also  \cite{Di,DHHR, DHN, KR, Sam1}).

\medskip

We now state the Sobolev embedding Theorem  (for the proofs see
\cite{DHHR}). Let,
\[
p^*(x):=
\begin{cases}
\frac{p(x)N}{N-p(x)} & \mbox{if } p(x)<N,\\
+\infty & \mbox{if } p(x)\ge N,
\end{cases}
\]
be the Sobolev critical exponent. Then we have the following,
\begin{teo}\label{embed}
Let $\Omega$ be a Lipschitz domain. Let $p:\Omega\to [1,\infty)$ and $p$ log-H\"{o}lder continuous. Then the imbedding $\wp\hookrightarrow \lpe$ is continuous.
\end{teo}

\section{$H^2-$Regularity for the non-degenerated problem for any  dimension}\label{nodeg}

In this section we assume that $\O$ is a bounded domain in $\R^N$, with $N\geq 2$.

\medskip

We want to study higher regularity of the weak solution of the regularized equation,
\begin{equation}
    \begin{cases}
        -\div\left((\varepsilon+|\nabla u|^2)^{\frac{p(x)-2}2}\nabla
        u\right) = f
        &\quad \mbox{in } \O,\\
        u=g &\quad \mbox{on } \partial\O,
        \end{cases}
    \label{regularizado}
\end{equation}
where $0<\ep\leq 1$, and $f\in \Lip$ and $g\in \wp$.

\medskip

The existence of a weak solution of \eqref{regularizado} holds by Theorem 13.3.3 in \cite{DHHR}.

\medskip
\begin{remark}\label{regularidadborde} Given $\ep\geq 0$, $p\in C^{\alpha_0}(\overline{\Omega})$ for some $\alpha_0>0$, and $g\in L^{\infty}(\O)$ we have the following results,
\begin{enumerate}

\item
Since $f,g\in L^{\infty}(\O)$,  by Theorem 4.1 in \cite{FZ}, we have that $u\in L^{\infty}(\O).$
\item  By Theorem 1.1 in \cite{FanGlobal},
 $u\in C_{loc}^{1,\alpha}(\O)$ for some $\alpha$ depending on $p_1,p_2,$ $\|u\|_{L^{\infty}(\O)}, \|f\|_{L^{\infty}(\O)}.$ Moreover, given
 $\O_0\subset\subset\O,$ $\|u\|_{C^{1,\alpha}(\O_0)}$ depends on the same constants and $dist(\O_0,\partial\O)$.
\item
 Finally, by Theorem 1.2 in \cite{FanGlobal}, if  $\partial\O\in C^{1,\gamma}$ and $g\in C^{1,\gamma}(\partial\O)$ for some $\gamma>0$ then
$u\in C^{1,\alpha}(\overline{\O})$, where  $\alpha$ and  $\|u\|_{C^{1,\alpha}(\O)}$ depend on $p_1,p_2,N,\|u\|_{L^{\infty}(\O)},$ $ \|p\|_{C^{\alpha_0}(\Omega)},\alpha_0,\gamma$.
\end{enumerate}
\end{remark}

We will first prove the  $H^2$-local regularity assuming only that $p(x)$ is Lipschitz. Then, we will prove
the global regularity under the stronger condition that
$\nabla p(x)$ is H\"{o}lder.

\subsection{$H^2-$Local regularity}

While we where finishing this paper, we found the work \cite{ChLy}, where the authors give a different proof of  the  $H^2$-local regularity  of the solutions of
\eqref{regularizado}. Anyhow, we leave the proof for the completeness of this paper.

\begin{teo}\label{regularidadloc}
Let $p, f\in \Lip$ with $p_1>1$  and $u$ a weak solution of \eqref{regularizado}, then
$u\in H^2_{loc}(\O).$
\end{teo}

\begin{proof}
First, let us define for any function $F$ and $h>0$,
$$
\Delta^h F(x)=\frac{F(x+{\bf h})-F(x)}{h},
$$
where ${\bf h}=he_k$ where $e_k$ is a vector of the canonical base of $\R^N.$

Let $\eta(x)=\xi(x)^2 \Delta^h u(x)$ where $\xi$ is a regular
function with compact support. Therefore, if we take
$v_\ep=
(|\nabla u|^2+\ep)^{1/2}$ and $h<\mbox{dist}(\mbox{supp}(\xi),\partial\O)$, we have
\begin{align*}
\int_{\O} \langle v_\ep(x)^{p(x)-2}\nabla u(x), \nabla \eta(x)\rangle\,
dx&=\int_{\O}f(x) \eta(x)\, dx\\
\int_{\O} \langle v_\ep(x+{\bf h})^{p(x+{\bf h})-2}\nabla u(x+{\bf h}), \nabla \eta(x)  \rangle \, dx&=\int_{\O}f(x+{\bf h}) \eta(x)\,
 dx.
\end{align*}
Subtracting, using that $\nabla \eta=2 \xi \nabla \xi \Delta^h
u+\xi^2 \Delta^h(\nabla u)$  and dividing by $h$ we obtain,
\begin{align*}
I=&\int_{\O} \langle \Delta^h(v_\ep(x)^{p(x)-2}  \nabla u),\Delta^h(\nabla u) \rangle \xi^2\, dx\\
=&-2\int_{\O} \langle \Delta^h(v_\ep(x)^{p(x)-2} \nabla u), \xi
\nabla \xi\Delta^h u \rangle\, dx +\int_{\O}
\xi^2 \Delta^h f
\Delta^h u\, dx\\
=&2\int_{\O} \left(\int_0^1(v_\ep(x+{\bf h}t)^{p(x+{\bf h}t)-2} \nabla u(x+{\bf h}t)\, dt\right)
\frac{\partial}{\partial x_k} (\xi \nabla \xi \Delta^h u) dx\\
&+\int_{\O} \xi^2 \Delta^h f \Delta^h u\, dx
\\
=& II+ III.
\end{align*}

Now, let as fix a ball $B_R$ such that $B_{3R}\subset\subset \O$ and take $\xi\in C^{\infty}_0(\O)$ supported in $B_{2R}$ such that
$0\leq \xi\leq 1,$ $\xi=1$ in $B_R$, $|\nabla \xi|\leq 1/R$ and $|D^2\xi|\leq C R^{-2}$.

\medskip

By Remark \ref{regularidadborde}, there exist a constant $C_1>0$ such that $|\nabla u|\leq C_1$ in $B_{3R}$, therefore we get
\begin{align*}
    II&\leq 2\int_{B_{2R}} \frac{C}{R} |\Delta^h u_{x_k}|\xi\, dx+2\int_{B_{2R}} \frac{C}{R^2} |\Delta^h u|\, dx\\
          &\leq\frac{C}{R}\int_{B_{2R}}  |\Delta^h (\nabla u)|\xi\, dx+C R^{N-2}.
\end{align*}

On the other hand, since $f$ is Lipschitz we have that,
\[
{|f(x+{\bf h})-f(x)|}\leq C_2 h
\]
for some constant $C_2>0$. This implies that,
$$
III\leq C_2 R^N.
$$
Therefore, summing $II$ and $III$,  and using Young's inequality, we have that for any $\delta>0$
\begin{equation}\label{primercota}
I\leq \delta \int_{B_{2R}} |\Delta^h (\nabla u)|^2\xi^2\, dx +C,
\end{equation}
for some constant $C$ depending on $R$ and $\delta$.

On the other hand observe that $I=I_1+I_2$ where,
$$
I_1=\frac1{h}\int_{B_{2R}}
\langle (v_\ep(x+{\bf h})^{p(x+{\bf h})-2} \nabla u(x+{\bf h})-v_\ep(x)^{p(x+{\bf h})-2} \nabla
u(x)), \Delta^h(\nabla u)\rangle \xi^2\,
 dx,$$
 and
$$
I_2=\frac1{h}\int_{B_{2R}}\langle\left(v_\ep(x)^{p(x+{\bf h})}-v_\ep(x)^{p(x)}\right)\frac{\nabla u(x)}{v_\ep(x)^2},\Delta^h(\nabla u) \rangle \xi^2\, dx.
$$

Using that $p(x)$ is Lipschitz and the fact that $|\nabla u(x)|\leq C_1$
we have that, for some $b$ between $p(x+h)$ and $p(x)$,
$$
\frac{1}{h}\left| v_\ep(x)^{p(x+{\bf h})}-v_\ep(x)^{p(x)}\right|= \left|v_\ep(x)^b \log(v_\ep(x))
\frac{p(x+{\bf h})-p(x)}{h}\right|\leq C,
$$
for some constant $C>0$ depending on $p_1,p_2,\ep, C_1$ and the Lipschitz constant of
$p(x)$.

Therefore, we have that
$$-I_2 \leq C C_1 \ep^{-1}  \int_{B_{2R}} |\Delta^h(\nabla u)|  \xi^2\, dx.$$

By \eqref{primercota}, the last inequality and using again Young's
inequality we have that, for any $\delta>0$

\begin{equation}\label{seundacota}
I_1\leq \delta \int_{B_{2R}} |\Delta^h (\nabla u)|^2\xi^2\, dx +C,
\end{equation}
for some constant $C>0$ depending on $p_1,p_2,\ep, C_1$ and the
Lipschitz constant of $p(x)$.

\medskip

To finish the proof, we have to find
a lower bound for $I_1$. By a well known inequality, we have that
\begin{align*}
    \langle (v_\ep(x+{\bf h})^{p(x+h)-2} \nabla u(x+{\bf h})-v_\ep(x)^{p(x+{\bf h})-2}
    \nabla u(x)), (\nabla u(x+{\bf h})-\nabla u(x))\rangle\\
    \geq C_{\ep} |\nabla u(x+{\bf h})-\nabla u(x)|^2 ,
\end{align*}
where
\[
C_{\ep}=\begin{cases}
    \ep^{\nicefrac{p(x+{\bf h})-2}{2}} & \mbox{  if } p(x+{\bf h})\ge 2, \\
    (p(x+{\bf h})-1)\ep^{\nicefrac{p(x+{\bf h})-2}{2}} & \mbox{  if } p(x+{\bf h})\le 2.
\end{cases}
\]
Therefore, using that $p_1> 1,$ we arrive at
$$
I_1\geq \int_{B_{2R}} C h^{-2}|\nabla
u(x+{\bf h})-\nabla u(x)|^2 \xi^2\, dx =  C\int_{B_{2R}} |\Delta^h(\nabla
u(x))|^2 \xi^2\, dx.$$

Finally combining the last inequality with \eqref{seundacota} we
have that,
$$
\int_{B_{R}} |\Delta^h(\nabla u(x))|^2 \, dx\leq C(N,p,f,\ep).
$$
This proves that $u\in H^{2}_{loc}(\O)$.
\end{proof}

\subsection{$H^2-$Global Regularity}

\medskip

Now we want to prove that if $f\in \Lip$ and $g\in
C^{1,\beta}(\partial\O),$ the regularized
equation \eqref{regularizado}
has a weak solution $u\in C^2(\Omega)\cap C^{1,\alpha}(\overline{\O})$ for
an $\alpha\in(0,1).$
We already know, by Remark \ref{regularidadborde}, that $u\in C^{1,\alpha}(\overline{\O})$. Then, we only need  to prove that  $u\in C^2(\Omega)$.



\medskip

%

\begin{lema}
 \label{lemaaux1}
Let $\O$ be a bounded domain in $\R^N$ with $\partial \O\in C^{1,\gamma},$ $p\in C^{1,\beta}(\Omega)\cap C^{\alpha_0}(\overline{\Omega})$, $f\in\Lip $
and $g\in C^{1,\beta}(\partial\O).$ Then,
the Dirichlet Problem \eqref{regularizado}
has a solution $u\in C^{2}(\O)\cap C^{1,\alpha}(\overline{\O}).$
\end{lema}

\begin{proof}
Observe that  by Theorem \ref{regularidadloc}, we know that the solution is in $H^2_{loc}(\O)$. Then for any $\O'\subset\subset\O$ we can derive the equation
 and look the solution of \eqref{regularizado} as the solution of the following equation,
\begin{equation}
    \begin{cases}
        L_\ep u=a(x) &\mbox{in } \O',\\
            u=u &\mbox{on } \partial\O'.
        \end{cases}
    \label{nodiv}
\end{equation}

Here,
\[
L_\ep u=a_{ij}^\ep (x)u_{x_ix_j}
\]
with
\begin{equation}\label{defa}
\begin{aligned}
 a_{ij}^\ep(x)&= \delta_{ij}+ (p(x)-2)\frac{u_{x_i} u_{x_j}}{v_\ep^2},\quad v_\ep=\left(\varepsilon+|\nabla u|^2\right)^{\frac12},\\
a_\ep(x)&=\ln(v_\ep)\langle \nabla u, \nabla p\rangle +fv_\ep^{2-p}.
\end{aligned}
\end{equation}

The operator $L_\ep$ is uniformly elliptic in $\Omega,$ since for any
$\xi\in\R^N$
\begin{equation}\label{el}
\min\{ (p_1 -1),1\}|\xi|^2\le a_{ij}^\ep\xi_i \xi_j\le \max\{ (p_2 -1),1\}|\xi|^2.
\end{equation}

On the other hand,  by Remark \ref{regularidadborde},  $u\in C^{1,\alpha}(\overline{\O})$. Then,  $a_{ij}^\ep\in C^{\alpha}(\overline{\O}) $, since $\ep>0$.
Using that $f\in \Lip,$ we have that $a\in
C^{\rho}(\Omega)$ where $\rho=\min(\alpha,\beta)$. If $\partial\O'\in C^2$, as $u$ is the unique solution of \eqref{nodiv},  by Theorem 6.13 in \cite{GT}, we have
 that  $u\in C^{2,\rho}(\O')$. This ends the proof.

\end{proof}

\begin{remark}\label{depp1}
        By the $H^2$ global estimate for linear elliptic equations with $L^{\infty}(\O)$ coefficients in two variables
    (see Lemma \ref{cotap11} and \eqref{el}) we have that,
    \[
         \|u\|_{H^2(\O)}\le C\left(\|a_\ep \|_{L^2(\O)} + \|g\|_{H^2(\O)}\right)
    \]
    where $u$ is the solution of \eqref{regularizado} and $C$ is a constant independents of $\ep.$
\end{remark}

\section {Proof of Theorem \ref{H2global} }\label{pt11}
Before proving the theorem, we will need a global bound for the derivatives of the solutions of \eqref{regularizado}.

\begin{lema}\label{lemaux2}
Let $f\in L^{q(x)}(\O)$ with $q'(x)\le p^*(x)$, $g\in \wp,$ $\ep>0$ and $u_{\ep}$ be the weak solution  of
 \eqref{regularizado} then
 \[
 \|\nabla u_{\ep}\|_{\lp}\leq C
 \]
   where $C$ is a constant depending on $\|f\|_{L^{q(\cdot)}(\O)},\|g\|_{\wp}$ but not on $\ep$.
\end{lema}
\begin{proof}
    Let
    \[
    J(v):=\int_{\O} \frac{1}{p(x)}(|\nabla v|^2+\ep)^{\nicefrac{p(x)}2}\, dx.
    \]
    By the convexity of $J$ and using \eqref{regularizado} we have that,
    \begin{align*}
        J(u_{\ep})&\leq J(g)-\int_{\O}(|\nabla u_{\ep}|^2+\ep)^{(p-2)/2} \nabla u_{\ep} (\nabla g-\nabla  u_{\ep})\, dx\\
        &\leq C\left(1+\int_{\O} f (u_{\ep}-g)\,dx\right) \\
        &\leq C\left(1+\|f\|_{L^{q(\cdot)}(\O)}\|u_{\ep}-g\|_{L^{q'(\cdot)}(\O)}\right)\\
        &\leq C\left(1+\|f\|_{L^{q(\cdot)}(\O)}\|\nabla u_{\ep}-\nabla g\|_{L^{p(\cdot)}(\O)}\right),
    \end{align*}
    where in the last inequality we are using that $\wp\hookrightarrow L^{p^*(\cdot)}(\O)$ continuously and Poincare's inequality.

    Thus we have that there exist a constant independent of $\ep$ such that,
    $$
    \int_{\O}|\nabla u_{\ep}|^{p(x)}\, dx \leq C(1+\|\nabla u_{\ep}\|_{\lp}),
    $$
    and using the properties of the $\lp-$ norms this means that
    $$
    \|\nabla u_{\ep}\|^m_{\lp} \leq C(1+\|\nabla u_{\ep}\|_{\lp}),
    $$
    for some $m> 1$. Therefore $\|\nabla u_{\ep}\|_{\lp}$ is bounded
    independent of $\ep$.
\end{proof}

\medskip

To prove Theorem \ref{H2global}, we will use the results of Section \ref{nodeg}. Therefore, we will first need to assume that
$p\in C^{1.\beta}(\Omega)\cap C(\overline{\Omega}).$

\begin{teo}\label{H2globalcbeta}
Let $\O$ be a bounded domain in $\R^2$ with $C^2$ boundary, $p\in
C^{1.\beta}(\Omega)\cap C^{\alpha_0}(\overline{\Omega})$ with
$p(x)\ge p_1>1$, $g\in H^2(\O)$ and $u$ be the weak solution  of
\eqref{problema}. If $f$ satisfies (F1) and (F2)
then $u\in H^2(\O).$
\end{teo}

\begin{proof}
Let $f_\ep\in\Lip$ and $g_\varepsilon\in
C^{2,\alpha}(\overline{\O})$ such that
    \begin{align*}
        f_\varepsilon\to f  &\mbox{ strongly in } \lq,\\
        g_\varepsilon\to g  &\mbox{ strongly in } H^2(\O),
    \end{align*}
    as $\varepsilon\to
    0.$ Observe that, since $f(x)=0$ if $p(x)>2,$ we can take $f_{\ep}\equiv 0$ in $\{x\in\O\colon p(x)>2\}$.

    Now, let us consider the solution of \eqref{regularizado} as the solution of
    \[
    \begin{cases}
        a_{11}^\ep(x)\frac{\partial^2 u_\varepsilon}{\partial x_1^2} + 2
        a_{12}^\ep(x)\frac{\partial^2 u_\varepsilon}{\partial x_1\partial
        x_2}+ a_{22}^\ep(x)
    \frac{\partial^2 u_\varepsilon}{\partial x_2^2}=a_\ep(x)
         &\mbox{in } \O,\\
        u_\varepsilon = g_\varepsilon &\mbox{on } \partial\O,
    \end{cases}
    \]
    where $a_{11}^\ep,a_{22}^\ep,a_{12}^\ep, a_\ep$ are defined as in Lemma \ref{lemaaux1}, substituting $f$ and $g$ by $f_{\ep}$
    and $g_{\ep}$ respectively.
    By Lemma \ref{lemaaux1} we know that $u_{\ep}\in C^2(\O)\cap C^{1,\alpha}(\overline{\O})$.

    First we will prove the $\{u_\ep\}Ì£_{\ep\in(0,1]}$ is bounded in $H^2(\O).$ By Remark \ref{depp1}, we have that
    \begin{equation}\label{cotaglobal}
        \begin{aligned}\|u_{\ep}\|_{H^2(\O)}&\leq C(\|a_\ep(x)\|_{L^2(\O)}+\|g_{\ep}\|_{H^2(\O)})\\
            &\leq C(\|\ln(v_\ep)\nabla u_{\ep}\nabla p\|_{L^2(\O)}+\|f_{\ep}v^{2-p}\|_{L^2(\O)}+\|g_{\ep}\|_{H^2(\O)}).
        \end{aligned}
    \end{equation}

    Taking $\O_1=\{x\in\O:|\nabla u_{\ep}(x)|>1\}$, using that $p(x)$ is Lipschitz and  H\"{o}lder's inequality, we have
        \begin{equation}\label{cotalog}
        \|\ln(v_\ep)\nabla u_{\ep}\nabla p\|_{L^2(\O)}\leq C \|\ln^2(v_\ep)\nabla u_{\ep}\|_{L^{p'(\cdot)}(\O_1)}^{\nicefrac{1}{2}}
        \|\nabla u_{\ep}\|_{L^{p(\cdot)}(\O_1)}^{\nicefrac{1}{2}} +C.
    \end{equation}

    On the other hand, since $q(x)\geq q_1>2$, we have that $q'(x)\leq p^*(x)$. Then, as $\|f_{\ep}\|_{\lq}$ and  $\|g_{\ep}\|_{H^2(\O)}$ are bounded independent of $\ep$,
    using  Lemma \ref{lemaux2} we conclude that $\|\nabla u_{\ep}\|_{\lp}$ is uniformly bounded.

    Observe that, for all $s>0$ there exist a constant $C>0$ such that
    $$
    \ln(v_\ep)\leq C v_\ep^{s/2}< C |\nabla u_{\ep}|^{s/2}\quad \mbox{ in } \O_1,
    $$
    thus
    \begin{align*}
    \|\ln^2(v_\ep)|\nabla u_{\ep}|\|_{L^{p'(\cdot)}(\O_1)}&\leq C \||\nabla u_{\ep}|^{1+s}\|_{L^{p'(\cdot)}(\O_1)}\\
    &\leq C \|\nabla u_{\ep}\|^{(1+s)}_{L^{p'(\cdot)(1+s)}(\O_1)}\\
    &\leq  C \| u_{\ep}\|^{(1+s)}_{{H^2}(\O_1)}.
    \end{align*}
    In the last line, we are using that $2^*=\infty$, since $N=2$.

    Then, by the last inequality, \eqref{cotaglobal} and \eqref{cotalog}, we get
    \begin{equation}\label{casiultima}
    \|u_{\ep}\|_{H^2(\O)}\leq C \left(\|u_{\ep}\|^{\nicefrac{(1+s)}{2}}_{H^2(\O)}+\|f_{\ep}v_\ep^{2-p}\|_{L^2(\O)}+1\right).
    \end{equation}

    Taking
    $$
    A_1=\{x\in\O: p(x)=2\}\quad \mbox{ and }\quad A_2=\{x\in\O: p(x)<2\}
    $$
    and using that $f_{\ep}\equiv 0$ in $\{x\in\O: p(x)>2\}$, we have that
    $$
    \|f_{\ep}v_\ep^{2-p}\|_{L^2(\O)}\leq \|f_{\ep}\|_{L^2(A_1)}+\|f_{\ep}v_\ep^{2-p}\|_{L^2(A_2)}.
    $$
    Since  $\|f_{\ep}\|_{L^2(A_1)}$ is bounded, to prove that $\{u_\ep\}_{\ep\in(0,1]}$ is bounded in $H^2(\O)$, we only have to find a bound of $\|f_{\ep}v_\ep^{2-p}\|_{L^2(A_2)}$.

Let  as define in $A_2$ the function
\[
\widetilde{q}(x)=\begin{cases}
\frac{1}{2p(x)-3}+1 &\quad \mbox{ if }\quad
\frac{1}{q(x)}+\frac{3}{2}\leq p(x)<2,\\
\frac{q(x)}{2} +1 &\quad \mbox{ if }\quad
p(x)<\frac{1}{q(x)}+\frac{3}{2}.
\end{cases}
\]
It is easy to see that $2<\widetilde{q}(x)\leq q(x)$ for any $x\in A_2$.

On the other hand, let us denote $\mu(x)={\frac{2\widetilde{q}(x)}{\widetilde{q}(x)-2}}$ and $\gamma(x)=\mu(x)(2-p(x))$ then
$$1< 1+\frac{2}{ q_2}\leq \gamma(x) \leq \max\left\{2,2+\frac{8}{q_1-2}\right\}\quad \forall x\in A_2.$$

Now, using H\"{o}lder's inequality  with exponent $\widetilde{q}(x)/2$, we have
\begin{equation}\label{casiultima2}
    \|f_{\ep}v_\ep^{2-p}\|_{L^2(A_2)}\leq C \|f_{\ep}\|_{L^{\widetilde{q}(\cdot)}(A_2)} \|v_\ep^{2-p}\|_{L^{\mu(\cdot)}(A_2)}.
\end{equation}

Then, if $\|v_\ep\|_{L^{\gamma(\cdot)}(A_2)}\leq 1$ we have $\|v_\ep^{2-p}\|_{L^{\mu(\cdot)}(A_2)}\le1$ and since $\widetilde{q}(x)\le q(x)$ we get
\[
    \|f_{\ep}v_\ep^{2-p}\|_{L^2(A_2)}\leq C.
\]
 If $\|v\|_{L^{\gamma(\cdot)}(A_2)}\geq 1$ ,  we have
\begin{equation}\label{casiultima3}
    \|v_\ep^{2-p}\|_{L^{\mu(\cdot)}(A_2)}\leq \|v_\ep\|^{2-p_1}_{L^{\gamma(\cdot)}(A_2)}\leq C (1+\|\nabla u_{\ep}\|^{2-p_1}_{L^{\gamma(\cdot)}(A_2)}),
\end{equation}
where in the last inequality we are using that $\ep\leq 1.$

Since $2^*=\infty$ and  $1<\gamma_1 \leq \gamma(x)\leq
\gamma_2<\infty$ , by the Sobolev embedding inequality, we have that
$$
\|\nabla u_{\ep}\|^{2-p_1}_{L^{\gamma(\cdot)}(A_2)}\leq C \|u_{\ep}\|^{2-p_1}_{H^2(A_2)}\leq C \|u_{\ep}\|^{2-p_1}_{H^2(\Omega)}.
$$
Combining this last inequality with inequalities \eqref{casiultima3}, \eqref{casiultima2}, \eqref{casiultima} and the fact that $\widetilde{q}(x)\leq q(x),$ we get
$$
\|u_{\ep}\|_{H^2(\O)}\leq  C (\|u_{\ep}\|^{\nicefrac{(1+s)}{2}}_{H^2(\O)}+\|u_{\ep}\|^{2-p_1}_{H^2(\Omega)}+1).
$$

Finally, we get that for any $0<s<1$ there exist a constant $C=C(p,g,f,s)$ such that
\[
\|u_{\ep}\|_{H^2(\O)}\leq C.
\]
Then, there exist a subsequence still denoted
$\{u_\ep\}_{\ep\in(0,1]}$ and $u\in H^1(\Omega)$ such that
\begin{align*}
    u_{\ep}\to u &\mbox{ strongly in } H^1(\O),\\
    u_{\ep}\rightharpoonup u  &\mbox{ weakly in } H^2(\O),
\end{align*}
It is clear that $u$ satisfies the boundary condition.

Lastly, by Proposition 3.2 in \cite{BN}, there exist a constant $M>0$ independent of $\ep$ such that,
\begin{equation}\label{desig}
|(\ep+|\nabla u_{\ep}|^2)^{\frac{p(x)-2}{2}}\nabla u_{\ep}-(\ep+|\nabla u|^2)^{\frac{p(x)-2}{2}}\nabla u|\le M |\nabla(u_{\ep}-u)|^{p(x)-1}
\end{equation}
for all $  x\in \O.$ Then, passing to the limit in the weak formulation of \eqref{regularizado} and using the above inequality, we have that
$$\int_{\O} |\nabla u|^{p(x)-2} \nabla u\nabla \varphi\, dx =\int_{\O}  f \varphi\, dx$$
for any $\varphi\in C_0^{\infty}(\Omega)$. Therefore $u\in H^2(\Omega)$ and solves \eqref{problema}.
\end{proof}

\bigskip
Now, we are able  to prove the theorem.

\begin{proof}[{\bf Proof of Theorem \ref{H2global}}]
    First, we consider the case $p\in C^1(\overline{\O}).$
    Let $p_{\ep}\in C^{\infty}(\overline{\Omega})$ such that $p_{\ep}\to p$ in $C^1(\O)$.
    Now, we define
    \begin{equation}\label{fep}
    f_{\ep}(x)=\begin{cases}f(x) &\mbox{ if } p_{\ep}(x)\leq 2,\\
        0 &\mbox{ if } p_{\ep}(x) > 2.\end{cases}
    \end{equation}
    Observe that $f_{\ep}\to f$ in $\lq$ as $\ep\to 0$.

    Then, by Theorem \ref{H2globalcbeta}, the solution $u_{\ep}$  of \eqref{problema} (with $p_{\ep}$ and $f_{\ep}$ instead of $p$ and $f$) is bounded in $H^2(\O)$
    by a constant independent of $\ep$. Therefore, there exist a subsequence still denoted $\{u_{\ep}\}_{\ep\in(0,1]}$ and $u\in H^2(\O)$ such that
    \begin{equation}\label{conv}
        \begin{aligned}
            u_{\ep}\to u &\quad \mbox{ in } H^1(\O),\\
            u_{\ep}\rightharpoonup u &\quad \mbox{ weakly in } H^2(\O).
        \end{aligned}
    \end{equation}
    It remains to prove that $u$ is a solution of \eqref{problema}. Let $\varphi\in C^{\infty}_0(\O)$, then
    \begin{equation}\label{ultimo1}
        \begin{aligned}
            \int_{\O} f_\ep\varphi \, dx =& \int_{\O}|\nabla u_{\ep}|^{p_{\ep}(x)-2} \nabla u_{\ep}\nabla \varphi \, dx\\
            =&\int_{\O} |\nabla u_{\ep}|^{p(x)-2} \nabla u_{\ep}\nabla \varphi \, dx\\
            +&\int_{\O} (|\nabla u_{\ep}|^{p_{\ep}(x)-2}-|\nabla u_{\ep}|^{p(x)-2}) \nabla u_{\ep}\nabla \varphi \, dx .
        \end{aligned}
        \end{equation}
    Therefore, using that $H^2(\O)\hookrightarrow \wp$ compactly, we have that
    \begin{equation}\label{ultimo22}
        \int_{\O} |\nabla u_{\ep}|^{p(x)-2} \nabla u_{\ep}\nabla \varphi \, dx\to \int_{\O} |\nabla u|^{p(x)-2} \nabla u\nabla \varphi \, dx.
    \end{equation}

    On the other hand, we have
    \[
        |\nabla u_\ep(x)|^{p_{\ep}(x)-1}-|\nabla u_\ep(x)|^{p(x)-1}= |\nabla u_\ep(x)|^{b_\ep(x)} \log(|\nabla u_\ep(x)|)
        (p_{\ep}(x)-p(x)),
    \]
    where $b_\ep(x)=p_{\ep}(x)\theta+(1-\theta)p(x)-1$ for some $0<\theta<1$.
    Therefore,  using that $2^*=\infty$ and that $p_{\ep}\to p$ uniformly, we obtain
    \begin{equation}\label{uleq}
        \int_{\O} (|\nabla u_{\ep}|^{p_{\ep}(x)-2}-|\nabla u_{\ep}|^{p(x)-2}) \nabla u_{\ep}\nabla \varphi \, dx\to 0.
        \end{equation}
    Then, using that $f_{\ep}\to f$ in $\lq$,  \eqref{ultimo1}, \eqref{ultimo22} and the \eqref{uleq} we conclude that $u$ is a solution of \eqref{problema}.

    \medskip

    Now, we consider the case $p\in \Lipc.$ By Lemmas \ref{ext} and \ref{lipyc1} there exists $p_{\ep}\in C^1(\overline{\Omega})$ such that
    $|\Omega\setminus \Omega_0|<\ep$ where
    \[
    \O_0=\{x\in\O \colon p_\ep(x)= p(x) \mbox{ and } \nabla p_\ep(x)=\nabla p(x)\}.
    \]

    We define $f_\ep$ as in \eqref{fep}. Then, the solution $u_{\ep}$  of \eqref{problema} with $p_{\ep}$ and $f_{\ep}$ instead
    of $p$ and $f$ is bounded in $H^2(\O)$ by a constant independent of $\ep$.
    Therefore there exist a subsequence still denoted $\{u_{\ep}\}_{\ep\in(0,1]}$ and $u\in H^2(\O)$ satisfying \eqref{conv}.

    Lastly, we prove that $u$ is a solution of \eqref{problema}.
    Let $\varphi\in C^{\infty}_0(\O).$ By H\"{o}lder inequality, since $2^*=\infty$ and by (3) of Lemma \ref{lipyc1}
    we have
    \begin{align*}
        \int_{\O\setminus\O_0}& (|\nabla u_{\ep}|^{p_{\ep}(x)-2}-|\nabla u_{\ep}|^{p(x)-2}) \nabla u_{\ep}\nabla \varphi \, dx \\ &
        \leq C(\|\nabla u_{\ep}\|_{L^{p_{\ep}}(\O)} \|1\|_{L^{p_{\ep}}(\O\setminus\O_0)}+\|\nabla u_{\ep}\|_{L^{p}(\O)} \|1\|_{L^{p}(\O\setminus\O_0)})
        \\ &\leq C\|u_{\ep}\|_{H^2(\O)}( \|1\|_{L^{p_{\ep}}(\O\setminus\O_0)}+ \|1\|_{L^{p}(\O\setminus\O_0)}).
    \end{align*}
    Then, since $\|u_{\ep}\|_{H^2(\O)}$ is bounded independent of $\ep$ and $|\O\setminus\O_0|<\ep$ we obtain that
    \[
    \int_{\O\setminus\O_0} (|\nabla u_{\ep}|^{p_{\ep}(x)-2}-|\nabla u_{\ep}|^{p(x)-2}) \nabla u_{\ep}\nabla \varphi \, dx\to 0.
    \]
    Therefore,  since  \eqref{ultimo1}, \eqref{ultimo22} again  hold,  using that $f_{\ep}\to f$ in $\lq,$ and the above equation,
    we conclude that $u$ is a solution of \eqref{problema}.
\end{proof}

\section{The convex case}\label{convex}

Lastly, we want to prove that the solution is in $H^2(\O)$ if we
only assume that  $\partial\O$ is convex. We want to remark here
that this result generalize the one in Theorem 2.2 in \cite{LB} in
two ways. In that paper the authors consider the case $p=constant$
and $g=0$. Instead, we are allowed to cover the case where $g$ is
any function in $H^2(\O)$ and $p(x)\in \Lipc$.

\begin{remark}\label{domain}
Let $\O$ be a convex set and $p:\O\to[1,\infty)$  be $\log-$continuous in $\overline{\O}$. Then, there exists
a sequence  $\{\O_m\}_{m\in\mathbb{N}}$ of convex subset of $\O$ with $C^2$ boundary such that $\O_m\subset\O_{m+1}$ for any $m\in\mathbb{N}$ and
$|\O\setminus\O_m|\to 0.$
\begin{enumerate}
\item
Then, there exists a constant $C$ depending on $p(x), |\O| $ such that
\[
\|v\|_{L^{p(\cdot)}(\O_m)}\le C \|\nabla v\|_{L^{p(\cdot)}(\O_m)} \quad\forall v\in W^{1,p(\cdot)}_0(\O_m),
\]
for any $m\in\mathbb{N}.$ This follows by Theorem 3.3 in \cite{KR}, using that $\O_m\subset\O_{m+1}$ for any $m\in\mathbb{N}.$

\item The Lipschitz constants of $\O_m$ ($m\in\mathbb{N}$) are uniformly bounded (see Remark 2.3 in \cite{LB}).
Therefore, the extension operators
\[
E_{1,m}: W^{1,p(\cdot)}(\O_m) \to W^{1,p(\cdot)}(\O) \mbox{ and } E_{2,m}:H^2(\O_m)\to H^2(\O)
\]
define as Theorem 4.2 in \cite{D3} satisfy that  $\|E_{1,m}\|$ and
$\|E_{2,m}\|$ are uniformly bounded.
\item By (2) and  Corollary 8.3.2 in \cite{DHHR},
there exists a constant $C$ independent of $m$ such that
\[
\|v\|_{L^{p^*(\cdot)}(\O_m)}\le C \|v\|_{W^{1,p(\cdot)}(\O_m)} \quad\forall v\in W^{1,p(\cdot)}(\O_m),
\]
for any $m\in\mathbb{N}.$
\end{enumerate}
We want to remark that all the constants of the above inequalities
are independent of $p_1$ (see Section \ref{coments} for the
applications).
\end{remark}

\begin{proof}[{\bf Proof of Theorem \ref{H2convexa}}]
    We begin taking $\{\O_m\}_{m\in\mathbb{N}}$ as in Remark \ref{domain} and  $u_m$  the solution of
\[
    \begin{cases}
       - \Delta_{p(x)} u_m =f  & \mbox{ in } \O_m,\\
         u_m =g  & \mbox{ on } \partial \Omega_m.
    \end{cases}
\]
By Theorem \ref{H2global}, $u_m\in H^2(\Omega_m)$ for any $m\in\mathbb{N}$. Moreover, $u_m$ solves
\[
\begin{cases}
        L^m u_m=a_{ij}^m(x)u_{m,{x_ix_j}}=a^m(x) &\mbox{in } \O_m,\\
            u_m=g &\mbox{on } \partial\O_m,
        \end{cases}
\]
    with
\begin{align*}
 a_{ij}^m(x)&= \delta_{ij}+ (p(x)-2)\frac{u_{m,{x_i}}(x) u_{{m,x_j}}(x)}{|\nabla u_m(x)|^2},\\
a^m(x)&=\ln(|\nabla u_m(x)|)\langle \nabla u_m(x), \nabla p(x)\rangle +f(x)|\nabla u_m(x)|^{2-p(x)}.
\end{align*}

Then $v_m=u_m-g$ solves
\[
    \begin{cases}
        L^m v_m=-L^mg+a^m(x) &\mbox{in } \O_m,\\
            v_m=0 &\mbox{on } \partial\O_m.
        \end{cases}
\]
Thus, using that $v_m\in H^2(\O_m)\cap H^1_0(\Omega_m)$ and since the coefficients $a^m_{ij}(x)$ are bounded independent of $m$,
we can argue as in Theorem 2.2 in \cite{LB} and obtain,
\begin{equation}
    \begin{aligned}
    &\|v_m\|_{H^2(\O_m)}\leq C\|-L^m g + f|\nabla u_m|^{2-p(\cdot)}+\ln(|\nabla u_m|)|\nabla u_m| \|_{L^2(\O_m)}\\
    &\leq C\left(\||\nabla u_m|^{2-p(\cdot)}\|_{L^2(\O_m)}+\|\ln(|\nabla u_m|)|\nabla u_m| \|_{L^2(\O_m)} +1 \right)
\end{aligned}
\label{cotavm}
\end{equation}
where the constant $C$ is independent of $m.$

As in Lemma \ref{lemaux2} we can prove, using Remark \ref{domain}
(1) and (3), that the norms $\|\nabla u_m\|_{L^{p(\cdot)}(\O_m)}$
are uniformly bounded. Therefore, proceeding as in Theorem \ref{H2globalcbeta} we obtain
\begin{equation}
    \begin{aligned}
    &\|\ln(|\nabla u_m|)|\nabla u_m| \|_{L^2(\O_m)}+\|f|\nabla u_m|^{2-p}\|_{L^2(\O_m)}\\
    &\leq C\left(\|\nabla u_m\|^{(1+s)/2}_{L^{p'(\cdot)(1+s)}(\O_{1,m})}+ \|\nabla u_m\|^{2-p_1}_{L^{\gamma(\cdot)}(A_{2,m})}+1\right),
\end{aligned}
\label{cotaum}
\end{equation}
with $C$ independent of $m,$ where
\[
\O_{1,m}=\{x\in\O_m:|\nabla u_{m}(x)|>1\}\mbox{ and }A_{2,m}=\{x\in\O_m: p(x)<2\}.
\]
Now, using Remark \ref{domain} (3) and (2), we have that for any $r>1$ that
\begin{equation}
    \label{remarkliu}
    \begin{aligned}
    \|v_m\|_{W^{1,r}(\O_m)}&\leq \|E_{2,m} v_m\|_{W^{1,r}(\O)}\\
    &\leq C \|E_{2,m} v_m\|_{H^2(\O)}\\
    &\leq C \|v_m\|_{H^2(\O_m)}
\end{aligned}
\end{equation}
where $C$ is independent of $m$.

Therefore,  using \eqref{cotavm},  \eqref{cotaum} and
\eqref{remarkliu}, we get
\begin{align*}
    \|v_m\|_{H^2(\O_m)}\leq & C (\| v_m\|^{(1+s)/2}_{H^2(\O_m)}+ \| v_m\|^{2-p_1}_{H^{2}(\O_m)}+\| g\|^{(1+s)/2}_{H^2(\O_m)}+ \| g\|^{2-p_1}_{H^{2}(\O_m)}+1)
\\ \leq & C(\| v_m\|^{(1+s)/2}_{H^2(\O_m)}+ \| v_m\|^{2-p_1}_{H^{2}(\O_m)}+1),
\end{align*}
where the constant $C$ is independent of $m$. This proves that $\{\|v_m\|_{H^2(\O_m)}\}_{m\in\mathbb{N}}$ is bounded.

Now we have, as in the proof of Theorem 2.2 in \cite{LB}, that there exist a subsequence still denote $\{v_m\}_{m\in\mathbb{N}}$ and a function
$v\in H^2(\O)\cap H^1_0(\O)$ such that,
\[
v_m \to v \quad\mbox{ strongly in } H^1(\O')
\]
for any $\O'\subset\subset \O.$ Then $u=v+g\in H^2(\O)$ and
\[
u_m \to u \quad \mbox{ strongly in } H^1(\O')
\]
for any $\O'\subset\subset \O.$
Thus, using \eqref{desig}, we have
\begin{equation}\label{limitem}
    |\nabla u_m|^{p(x)-2}\nabla u_m  \to  |\nabla u|^{p(x)-2}\nabla u \quad \mbox{ strongly in } L^{p'(\cdot)}(\O')
\end{equation}
for any $\O'\subset\subset \O.$

On the other hand, for any  $\varphi\in C^{\infty}_0(\O)$ there exist $m_0$ such that for all $m\geq m_0$
\[
\int_{\Omega_m} |\nabla u_m|^{p(x)-2}\nabla u_m \nabla \varphi\, dx=\int_{\Omega_m} f \varphi\, dx.
\]
Therefore, using \eqref{limitem} we have that $u$ is a weak solution of \eqref{problema}.
\end{proof}

\medskip
\begin{proof}[{\bf Proof of Corollary \ref{c1alpha}}]

By the previous theorem we have that $u\in H^2(\O)$, then we can derive the equation \eqref{problema} and obtain
\[
    \begin{cases}
        -a_{ij}(x)u_{{x_ix_j}}=a(x) &\mbox{in } \O,\\
            u=g &\mbox{on } \partial\O,
        \end{cases}
\]
where
\begin{align*}
 a_{ij}(x)&= \delta_{ij}+ (p(x)-2)\frac{u_{{x_i}}(x) u_{{x_j}}(x)}{|\nabla u(x)|^2},\\
a(x)&=\ln(|\nabla u(x)|)\langle \nabla u(x), \nabla p(x)\rangle +f(x)|\nabla u(x)|^{2-p(x)}.
\end{align*}

Using that $f\in \lq$ with $q(x)\geq q_1>2$ and following the lines in the proof of Theorem \ref{H2globalcbeta},
we have that $a(x)\in L^s(\O)$ with $s>2$. Therefore, by Remark \ref{c1alphalineal}, we have that $u\in C^{1,\alpha}(\overline{\O})$.
\end{proof}

\section{Comments}\label{coments}
In the image processing problem it is of interest the case where
$p_1$ is close to 1. By this reason, we are also interested in the
dependence of the $H^2-$norm on $p_1.$

If $N=2,$  $g\in H^2(\O)$ and $u_\ep$ is the solution of \eqref{regularizado}, we have by Lemma \ref{cotap11}, \eqref{defa} and \eqref{el}, that
there exists a constant $C$ independent of $p_1$ and $\ep$ such that
    \[
    \|u_\ep\|_{H^2(\O)}\le \frac{C}{(p_1-1)^{\kappa}}\left(\|a_\ep\|_{L^2(\O)} + \|g\|_{H^2(\O)}\right),
    \]
where $\kappa=1$ if $\O$ is convex and $\kappa=2$ if $\partial\O\in
C^2.$ Therefore, using that the Poincare's inequality and the
embedding $\wp\hookrightarrow\lpe$ hold in the case $p_1=1$ and
following the lines of Theorem \ref{H2global} and Theorem
\ref{H2convexa} we have that
\[
\|u\|_{H^2(\O)}\le \frac{C}{(p_1-1)^{\kappa}},
\]
where the constant $C$ is independent of $p_1.$

\appendix

\section{Regularity results for elliptic linear equations with coefficients in $L^{\infty}$}\label{B}
\setcounter{equation}{0}
\renewcommand{\theequation}{A.\arabic{equation}}

Let $\O$ be an bounded open subset of $\R^2$ and
\[
\M u=a_{ij} (x)u_{x_ix_j},
\]
such that $a_{ij}=a_{ji}$ and for any $\xi\in\R^N$
\begin{equation}\label{el1}
\lambda|\xi|^2\le a_{ij}(x)\xi_i \xi_j\le \Lambda|\xi|^2,
\end{equation}
and
\begin{equation}\label{el2}
    M_1 \le a_{11}(x)+a_{22}(x)\le M_2 \quad \mbox{in }\O
\end{equation}
where $\lambda,\Lambda, M_1$ and $M_2$ are positive constant.

\medskip

In the next lemma, we will give a $H^2-$bound for solutions of
\begin{equation}
    \begin{cases}
        \M u=f &\mbox{in } \O,\\
        u = g &\mbox{on }\partial\Omega,
    \end{cases}
    \label{eliprob}
\end{equation}
In fact, the following result is proved in Theorem 37,III in
\cite{Mi}, but it is not explicit the dependence of the bounds on
the ellipticity and the $L^\infty-$norm of $(a_{ij}(x)).$ Then,
following the proof of the mentioned theorem we can prove

\begin{lema}
    \label{cotap11}
    Let $\O$ be a bounded domain in $\R^2,$ $f\in L^{2}(\O)$ and $g\in H^2(\O).$ Then, if $u$ is a solution  of \eqref{eliprob}
         and $u\in H^2(\O)$ we have that
    \[
    \|u\|_{H^2(\O)}\le \frac{C}{\lambda^{\kappa}}\left(\|f\|_{L^2(\O)} + \|g\|_{H^2(\O)}\right),
    \]
where $\kappa=1$ if $\O$ is convex and $\kappa=2$ if $\partial \O\in C^2$ and $C$ is a constant independent of $\lambda$.
\end{lema}

\begin{proof}
    In this proof, we denote $u_{ij}=u_{x_ix_j}$  for all $i,j=1,2$ and $C$ is a constant independent of $\lambda.$

First, we consider the case $g\equiv0.$ Using \eqref{el1}, we have that
    \begin{align*}
        (a_{11}(x)+ a_{22}(x))(u_{12}^2 -u_{11}u_{22})=&\sum_{i,j,k=1}^2 a_{ij} u_{ki}u_{kj}-\Delta u
        \sum_{ij=1}^2 a_{ij} u_{ij}\\
        \ge&\lambda\sum_{ik=1}^2  u_{ki}^2-\Delta u f(x).
    \end{align*}
    Then, using Young's inequality, we get
    \begin{equation*}
        \frac{\lambda}{2(a_{11}(x) + a_{22}(x))}\sum_{ik=1}^2  u_{ki}^2\le
        \frac{4}{\lambda(a_{11}(x) + a_{22}(x))} f(x)^2
        + u_{12}^2 -u_{11}u_{22},
        \end{equation*}
    and by \eqref{el2}, we have that
    \begin{equation}\label{eq11}
        \sum_{ik=1}^2  u_{ki}^2\le
        \frac{C}{\lambda^2} f(x)^2
        +\frac{C}{\lambda} (u_{12}^2 -u_{11}u_{22}),
        \end{equation}

    Now, using (37.4) and (37.6) in \cite{Mi}, we have that for any $u\in H^2(\Omega)$
    \begin{equation}\label{eq21}
        \int_\O (u_{12}^2 -u_{11}u_{22}) \ dx =- \int_{\partial\O} \left( \frac{\partial u}{\partial\nu} \right)^2 \frac{H}{2} ds
    \end{equation}
        where $H$ is the curvature of $\partial\O$. If $\O$ is convex, then $H\ge 0$ and therefore, using \eqref{eq11} and \eqref{eq21} we have that
         \begin{equation}\label{eq221}
      \|D^2 u\|_{L^2(\Omega)} \le \frac{C}{\lambda}\|f\|_{L^2(\O)}.
      \end{equation}

        In the general case, we can use the following inequality
            \begin{equation}\label{eq31}
    \int_{\partial\O} \left( \frac{\partial u}{\partial\nu} \right)^2 ds\le C\left( (1+\delta^{-1}) \int_\O |\nabla u|^2 \ dx + \delta  \int_\O \sum_{ik=1}^2  u_{ki}^2
     \ dx  \right)
    \end{equation}
     for any $\delta>0.$ See equation (37.6) of \cite{Mi}.

     Then, by \eqref{eq11}, \eqref{eq21}, using that $H$ is bounded and \eqref{eq31} (choosing $\delta$ properly) we arrive to
     \begin{equation}\label{eq41}
     \int_\O \sum_{ik=1}^2  u_{ki}^2 \ dx\le \frac{C}{\lambda^2}\left(\int_\O f(x)^2 \ dx + \int_\O |\nabla u|^2 \ dx\right).
          \end{equation}

      On the other hand, using that $L u =f $ in $\O,$ \eqref{el1} and the Poincare's inequality, we have
      \begin{equation}
          \|\nabla u\|_{L^2(\O)}\le \frac{C}{\lambda}\|f\|_{L^2(\O)}.
          \label{eq51}
      \end{equation}

      Therefore, by \eqref{eq41} and \eqref{eq51}, we get
      \[
      \|D^2 u\|_{L^2(\Omega)} \le \frac{C}{\lambda^2}\|f\|_{L^2(\O)}.
      \]

      Thus, by the last inequality, \eqref{eq51} and \eqref{eq221} the lemma is proved in the case $g=0$.

When $g$ is any function in $H^2(\O)$ the lemma follows taking $v=u-g$.
     \end{proof}

The following theorem is proved in Corollary 8.1.6 in \cite{Gv}.
\begin{teo}
Let $\O$ be a convex  polygonal domain in $\R^2$, $\M$ satisfying \eqref{el1} and $u\in H^2(\O)\cap H^1_0(\O)$ be a solution of $\eqref{eliprob}$ with $g=0$ and $f\in L^p(\O)$ with $p>2$. Then $\nabla u\in C^{\mu}(\overline{\O})$ for some $0<\mu<1$.
\end{teo}

\begin{remark}\label{c1alphalineal}
Observe that the above Theorem holds also if we consider any $g\in W^{2,p}(\O)$, since we can take  $v=u-g$ in \eqref{eliprob} and use that $W^{2,p}(\O)\hookrightarrow C^{1,1-2/p}(\overline{\O})$ .
\end{remark}

\section{Lipschitz Functions}\label{C}

\renewcommand{\theequation}{C.\arabic{equation}}

Using the linear extension operator define in \cite{ER}, we have the following lemma

\begin{lema}\label{ext}
Let $\O$ be a bounded open domain with Lipschitz boundary and $f\in\Lipc.$ Then, there exists a
function $\overline{f}:\R^N\to\R$ such that $\overline{f}$ is a Lipschitz function,
$\sup_{\R^N}\overline{f}=\inf_{\overline{\O}}f$ and $\inf_{\R^N}\overline{f}=\max_{\overline{\O}} f.$
\end{lema}

\begin{lema}
    \label{lipyc1}
    Let $f:\R^N\to\R$ be Lipschitz function. Then for each $\ep>0,$ there exists a $C^1$ function
    $f_\ep:\R^N\to\R$ such that
\begin{enumerate}
\item
    $
    |\{x\in\R^N\colon f_\ep(x)\neq f(x) \mbox{ or } Df_\ep(x)\neq Df(x)\}| \le \ep.
    $
    \item
There exist a constant $C$ depending only on $N$ such that,
    \[
    \|Df_\ep\|_{L^\infty(\R^N)}\le C Lip(f).
    \]

\item
If $1<f_1\le f(x)\le f_2 $ in $\R^N,$ we have
    \[
    1 < f_\ep(x)\le f_2 + C\ep^{\frac1{N}} \mbox{ in } \R^N
    \]
    with $C$ a constant depending  only on $N.$
\end{enumerate}
\end{lema}
\begin{proof}
Items (1) and (2) follow by Theorem 1, pag. 251 in \cite{EG}.

To prove (3), let as define $$\Omega_0=\{x\in\R^N\colon f_\ep(x)= f(x) \mbox{ and } Df_\ep(x)=Df(x)\}$$
and  let as suppose that there exist $x\in\R^N\setminus\Omega_0$  such that $f_{\ep}(x)=f_2+\delta$ with $\delta>0$. If $x_0\in \Omega_0$,  by (2), we have
\[C Lip(f) |x-x_0|\geq f_{\ep}(x)-f_{\ep}(x_0)=f_2+\delta-f(x_0)\ge \delta.\]
Then $B_{\rho}(x)\subset \R^N\setminus\Omega_0$ where $\rho={\delta}(C Lip(f))^{-1}$ and using (1) we get $\delta\le C \ep^{1/N}$, for some constant $C$ independent of $\ep$.

Analogously we can prove the other inequality.
\end{proof}

\bibliographystyle{amsplain}

\begin{thebibliography}{10}



\bibitem{BN}
 Emilio Acerbi and Giuseppe Mingione , \emph{Regularity results for a class of functionals with
              non-standard growth},Arch. Ration. Mech. Anal. \textbf{156} (2001), no.~2, 121--140.

\bibitem{BN}
Jacques Baranger and Khalid Najib, \emph{Analyse num\'erique des \'ecoulements
  quasi-newtoniens dont la viscosit\'e ob\'eit \`a\ la loi puissance ou la loi
  de carreau}, Numer. Math. \textbf{58} (1990), no.~1, 35--49.

\bibitem{BCE}
Erik~M. Bollt, Rick Chartrand, Selim Esedo{\=g}lu, Pete Schultz, and Kevin~R.
  Vixie, \emph{Graduated adaptive image denoising: local compromise between
  total variation and isotropic diffusion}, Adv. Comput. Math. \textbf{31}
  (2009), no.~1-3, 61--85.

\bibitem{ChLy}
S.~Challal and A.~Lyaghfouri, \emph{Second order regularity for the
  p(x)-laplace operator}, Math. Nachr. \textbf{284} (2011), no.~10,
  1270–1279.

\bibitem{CLR}
Yunmei Chen, Stacey Levine, and Murali Rao, \emph{Variable exponent, linear
  growth functionals in image restoration}, SIAM J. Appl. Math. \textbf{66}
  (2006), no.~4, 1383--1406 (electronic).

\bibitem{Ci}
Ph. Ciarlet, \emph{The finite element method for elliptic problems}, vol.~68,
  North-Holland, Amsterdam, 1978.

\bibitem{CM}
A. Coscia and  G. Mingione \emph{Holder continuity of the gradient of p(x) harmonic
mappings}, Comptes Rendus de l'Académie des Sciences, Ser. I,
Mathematique 328 (1999) 363-368.

\bibitem{DLM}
L.~M. {Del Pezzo}, A.~{Lombardi}, and S.~{Mart{\'{\i}}nez}, \emph{{IP-DGFEM
  method for the \$p(x)\$- Laplacian}}, Preprint.
  http://arxiv.org/abs/1009.2063.

\bibitem{Di}
L~Diening, \emph{Theoretical and numerical results for electrorheological
  fluids,}, Ph.D. thesis, University of Freiburg, Germany (2002).

\bibitem{D}
L.~Diening, \emph{Maximal function on generalized {L}ebesgue spaces {${L}\sp
  {p(\cdot)}$}}, Math. Inequal. Appl. \textbf{7} (2004), no.~2, 245--253.

\bibitem{D3}
\bysame, \emph{Riesz potential and {S}obolev embeddings on generalized
  {L}ebesgue and {S}obolev spaces {$L^{p(\cdot)}$} and {$W^{k,p(\cdot)}$}},
  Math. Nachr. \textbf{268} (2004), 31--43.

\bibitem{DHHR}
L.~Diening, P~Harjulehto, P.~H\"{a}st\"{o}, and M.~Ruzicka, \emph{Lebesgue and
  sobolev spaces with variable exponents}, Lecture Notes in Mathematics, vol.
  2017, Springer-Verlag, New York, 2011.

\bibitem{DHN}
L.~Diening, P.~H{\"a}st{\"o}, and A.~Nekvinda, \emph{Open problems in variable
  exponent {L}ebesgue and {S}obolev spaces}, Function Spaces, Differential
  Operators and Nonlinear Analysis, Milovy, Math. Inst. Acad. Sci. Czech
  Republic, Praha, 2005.

\bibitem{ER}
David~E. Edmunds and Ji{\v{r}}{\'{\i}} R{\'a}kosn{\'{\i}}k, \emph{Sobolev
  embeddings with variable exponent}, Studia Math. \textbf{143} (2000), no.~3,
  267--293.

\bibitem{ELM} L. Esposito, F. Leonetti and  G. Mingione \emph{Sharp regularity for functionals
with (p,q) growth} Journal of Differential Equations 204 (2004) 5-55.


\bibitem{EG}
L.~C. Evans and R.~F. Gariepy, \emph{Measure theory and fine properties of
  functions}, Studies in Advanced Mathematics, CRC Press, Boca Raton, FL, 1992.

\bibitem{FanGlobal}
Xianling Fan, \emph{Global {$C^{1,\alpha}$} regularity for variable exponent
  elliptic equations in divergence form}, J. Differential Equations
  \textbf{235} (2007), no.~2, 397--417.

\bibitem{FZ}
Xianling Fan and Dun Zhao, \emph{A class of {D}e {G}iorgi type and {H}\"older
  continuity}, Nonlinear Anal. \textbf{36} (1999), no.~3, Ser. A: Theory
  Methods, 295--318.

\bibitem{GT}
D.~Gilbarg and N.~S. Trudinger, \emph{Elliptic partial differential equations
  of second order}, Grundlehren der Mathematischen Wissenschaften [Fundamental
  Principles of Mathematical Sciences], vol. 224, Springer-Verlag, Berlin,
  1983.

\bibitem{Gv}
P.~Grisvard, \emph{Elliptic problems in nonsmooth domains}, Monographs and
  Studies in Mathematics, vol.~24, Pitman (Advanced Publishing Program),
  Boston, MA, 1985.

\bibitem{KR}
Kov\'a\v{c}ik and R\'akosn{\'i}k, \emph{On spaces ${L}^{p(x)}$ and
  ${W}^{k,p(x)}$}, Czechoslovak Math. J \textbf{41} (1991), 592--618.

\bibitem{LB}
W.~B. Liu and John~W. Barrett, \emph{A remark on the regularity of the
  solutions of the {$p$}-{L}aplacian and its application to their finite
  element approximation}, J. Math. Anal. Appl. \textbf{178} (1993), no.~2,
  470--487.

\bibitem{Mi}
Carlo Miranda, \emph{Partial differential equations of elliptic type},
  Ergebnisse der Mathematik und ihrer Grenzgebiete, Band 2, Springer-Verlag,
  New York, 1970, Second revised edition. Translated from the Italian by Zane
  C. Motteler.

\bibitem{R}
Michael R{\ocirc{u}}{\v{z}}i{\v{c}}ka, \emph{Electrorheological fluids:
  modeling and mathematical theory}, Lecture Notes in Mathematics, vol. 1748,
  Springer-Verlag, Berlin, 2000.

\bibitem{Sam1}
S.~Samko, \emph{Denseness of {$C\sp \infty\sb 0(\bold R\sp N)$} in the
  generalized {S}obolev spaces {$W\sp {M,P(X)}(\bold R\sp N)$}}, Direct and
  inverse problems of mathematical physics ({N}ewark, {DE}, 1997), Int. Soc.
  Anal. Appl. Comput., vol.~5, Kluwer Acad. Publ., Dordrecht, 2000,
  pp.~333--342.

\end{thebibliography}
\def\cprime{$'$} \def\ocirc#1{\ifmmode\setbox0=\hbox{$#1$}\dimen0=\ht0
  \advance\dimen0 by1pt\rlap{\hbox to\wd0{\hss\raise\dimen0
  \hbox{\hskip.2em$\scriptscriptstyle\circ$}\hss}}#1\else {\accent"17 #1}\fi}
\providecommand{\bysame}{\leavevmode\hbox to3em{\hrulefill}\thinspace}
\providecommand{\MR}{\relax\ifhmode\unskip\space\fi MR }
\providecommand{\MRhref}[2]{%
  \href{http://www.ams.org/mathscinet-getitem?mr=#1}{#2}
}
\providecommand{\href}[2]{#2}

\end{document}